\documentclass[12pt,twoside]{amsart}

\usepackage{amssymb}
\usepackage{graphicx}
\usepackage[pdfauthor={David Cook II},
            pdftitle={Nested colourings of graphs},
            pdfsubject={Combinatorics: Graph theory},
            pdfkeywords={Proper vertex colourings, chromatic number, nested neighbourhoods},
            pdfproducer={LaTeX with hyperref},
            pdfcreator={latex->dvips->ps2pdf},
            pdfborder={0 0 0},
            colorlinks=false]{hyperref}
\usepackage[margin=1in]{geometry}

\def\urltilda{\kern -.15em\lower .7ex\hbox{\~{}}\kern .04em}  
\newcommand{\mC}{\ensuremath{\mathcal{C}}}                    
\newcommand{\st}{\ensuremath{\colon\,}}                       
\newcommand{\dcup}{\ensuremath{\,\mathaccent\cdot\cup\,}}     
\newcommand{\ddd}{\ensuremath{\dcup \cdots \dcup}}            
\newcommand{\gprod}{\ensuremath{\,\square\,}}                 
\DeclareMathOperator{\Hom}{Hom}                               
\DeclareMathOperator{\Ind}{Ind}                               
\DeclareMathOperator{\conn}{conn}                             

\numberwithin{figure}{section}
\numberwithin{equation}{section}

\newtheorem{theorem}{Theorem}[section]
\newtheorem{lemma}[theorem]{Lemma}
\newtheorem{proposition}[theorem]{Proposition}
\newtheorem{corollary}[theorem]{Corollary}
\newtheorem{conjecture}[theorem]{Conjecture}

\theoremstyle{definition}
\newtheorem{definition}[theorem]{Definition}
\newtheorem{remark}[theorem]{Remark}
\newtheorem{example}[theorem]{Example}
\newtheorem{question}[theorem]{Question}
\newtheorem*{acknowledgement}{Acknowledgement}

\newcount\HOUR
\newcount\MINUTE
\newcount\HOURSINMINUTES
\newcount\INTVAL
\newcommand{\twodigit}[1]{\INTVAL=#1\relax\ifnum\INTVAL<10 0\fi\the\INTVAL}
\HOUR=\time\divide\HOUR by 60\relax
\HOURSINMINUTES=\HOUR\multiply\HOURSINMINUTES by 60\relax
\MINUTE=\time\advance\MINUTE by -\HOURSINMINUTES\relax
\newcommand\rightnow{
           \twodigit{\the\HOUR}:\twodigit{\the\MINUTE},
           \twodigit{\number\day}.\space
           \ifcase\month\or January\or February\or March\or April\or May\or June\or July\or August\or September\or October\or November\or December\fi
           \space\number\year}

\begin{document}


\title{Nested colourings of graphs}
\author[D.\ Cook II]{David Cook II}
\address{Department of Mathematics, University of Notre Dame, Notre Dame, IN 46556, USA}
\email{\href{mailto:dcook8@nd.edu}{dcook8@nd.edu}}

\begin{abstract}
    A proper vertex colouring of a graph is \emph{nested} if the vertices of each of its colour
    classes can be ordered by inclusion of their open neighbourhoods.  Through a relation to
    partially ordered sets, we show that the nested chromatic number can be computed in polynomial time.

    Clearly, the nested chromatic number is an upper bound for the chromatic number of a graph.
    We develop multiple distinct bounds on the nested chromatic number using common properties of
    graphs.  We also determine the behaviour of the nested chromatic number under several graph
    operations, including the direct, Cartesian, strong, and lexicographic product.  Moreover, we
    classify precisely the possible nested chromatic numbers of graphs on a fixed number of vertices
    with a fixed chromatic number.
\end{abstract}

\maketitle

\section{Introduction}\label{sec:introduction}

Let $G$ be a finite simple graph on vertex set $V(G)$ and with edge set $E(G)$.
A partition $\mC = C_1 \ddd C_k$ of the vertices is a proper vertex colouring of
$G$ if the $C_i$ are independent sets.  The chromatic number $\chi(G)$ is the least
cardinality of a proper vertex colouring of $G$.  It is well-known that computing
the chromatic number of a graph is $NP$-complete.

We define a novel colouring of a graph $G$.  In particular, a proper vertex
colouring of $G$ is \emph{nested} if the vertices of each of its colour classes can be
ordered by inclusion of their open neighbourhoods.  The \emph{nested chromatic number}
$\chi_N(G)$ is the least cardinality of a nested colouring of $G$.
The nested chromatic number clearly bounds the chromatic number from above.
Through a connection to partially ordered sets, we prove that computing the
nested chromatic number of a graph can be done in polynomial time (Theorem~\ref{thm:poly-time}).

The concept of a nested colouring is extended to simplicial complexes in~\cite{Co-UFI}.
Using Proposition~\ref{pro:clique-exchange} it is shown that the nested chromatic number of
the underlying graph of a simplicial complex bounds from  below the nested chromatic number
of the complex itself.  Moreover, a new face ideal for a simplicial complex with respect to a
colouring is defined therein.  It is the nested colourings of simplicial complexes that give
rise to ideals which have minimal linear resolutions supported on a cubical complex.

Herein we consider the properties of nested colourings and the nested chromatic number.
This note is organised as follows.  In Section~\ref{sec:nested} we introduce the relevant
new definitions.  In particular, we define the weak duplicate preorder which provides two
distinct interpretations of nested colourings (Propositions~\ref{pro:clique-exchange}
and~\ref{pro:chain-covers}).  Through the latter interpretation we see that the nested 
chromatic number is related to the Dilworth number of a graph (Remark~\ref{rem:dilworth}).

In Section~\ref{sec:families} we classify the structure of graphs
with nested chromatic number $2$ (Theorem~\ref{thm:bip}).  Further, we study the nested chromatic number of regular
graphs and diamond- and $C_4$-free graphs therein.  In Section~\ref{sec:behaviour} we consider the
behaviour of the nested chromatic number under many common operations, including: Mycielski's construction,
the disjoint union, the join, the direct product, the Cartesian product, the strong product, and the composition
or lexicographic product.  In Section~\ref{sec:exists} we provide a classification of the triples $(\#V(G), \chi(G), \chi_N(G))$
that can occur for some graph $G$ (Theorem~\ref{thm:possible-chi-chi-s}).

For standard definitions not given here and for more examples, we refer the reader to any
standard graph theory textbook, e.g., \cite{West}.

\section{Nested colourings}\label{sec:nested}

In this section, we introduce three new concepts:  nested colourings, the de-duplicate
graph, and the weak duplicate preorder.

\subsection{Nested colourings \& the nested chromatic number}\label{sub:nested}~

We first define a nested neighbourhood condition on vertices of a finite simple graph.
We use $N_G(u) = \{v \st \{u,v\} \in E(G)\}$ to denote the open neighbourhood of
$u$ in $G$ and $N_G[u] = N_G(u) \dcup \{u\}$ to denote the closed neighbourhood of $u$ in $G$.

\begin{definition}\label{def:dup}
    Let $G$ be a finite simple graph, and let $u, v$ be vertices of $G$.  The vertex $u$ is a
    \emph{weak duplicate} of $v$ if $N_G(u) \subset N_G(v)$; if equality holds, then $u$ is a
    \emph{duplicate} of $v$.  Further, a \emph{duplicate-free graph} is a finite simple graph
    for which no pair of vertices are duplicates.  An independent set $I$ of $G$ is
    \emph{nested} if the vertices of $I$ can be linearly ordered so that $v \leq u$ implies
    $u$ is a weak duplicate of $v$.
\end{definition}

The order on the vertices of a nested independent set is unique, up to permutations
of duplicates.  This will be formalised in Section~\ref{sub:order}.

Using this condition on the neighbourhoods, we define a novel proper vertex colouring
of a graph.

\begin{definition}\label{def:nested}
    Let $G$ be a finite simple graph, and let $\mC$ be a proper vertex $k$-colouring
    $C_1 \ddd C_k$ of $G$.  If every colour class of $\mC$ is nested, then $\mC$ is a
    \emph{nested colouring} of $G$.  The \emph{nested chromatic number} $\chi_N(G)$ is
    the least cardinality of a nested colouring of $G$.  Moreover, the graph $G$ is
    \emph{colour-nested} if $\chi_N(G) = \chi(G)$.
\end{definition}

\begin{example}\label{exa:first-example}
    \begin{figure}[!ht]
        \includegraphics[scale=0.75]{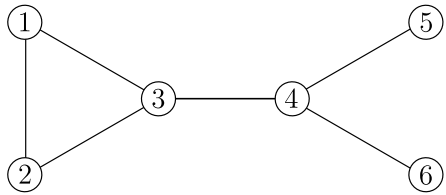}
        \caption{A graph $G$ with $\chi(G) = 3$ and $\chi_N(G) = 4$.}
        \label{fig:first-example}
    \end{figure}
    Let $G$ be the graph in Figure~\ref{fig:first-example}.  The partition
    $\{1,4\} \dcup \{2\} \dcup \{3,5,6\}$ is an optimal proper vertex $3$-colouring
    of $G$.  However, since $N_G(1) = \{2,3\}$ and $N_G(4) = \{3,5,6\}$, i.e.,
    the independent set $\{1, 4\}$ is not nested, the $3$-colouring is not nested.  
    Indeed, all proper vertex $3$-colourings of $G$ are not nested.  However, the
    proper vertex $4$-colouring $\{1\} \dcup \{2\} \dcup \{3, 5, 6\} \dcup \{4\}$ is 
    nested; indeed, $N_G(5) = N_G(6) = \{4\} \subset N_G(3) = \{1,2,4\}$.  Notice
    that the vertices $5$ and $6$ are duplicates.  Finally, as
    $\chi(G) = 3 < \chi_N(G) = 4$, we see that $G$ is not colour-nested.
\end{example}

We notice that isolated vertices are ``ignorable.''

\begin{remark}\label{rem:isolated-vertices}
    Isolated vertices are exactly those vertices that have an empty open neighbourhood.  Since the
    empty set is a subset of \emph{every} set, isolated vertices are weak duplicates of \emph{every}
    vertex of a graph.  Thus isolated vertices can be put in to \emph{any} colour class without
    modifying the nesting of the colour class.
\end{remark}

We also have a pair of immediate bounds on the nested chromatic number.

\begin{remark}\label{rem:singleton}
    Since every nested colouring of a finite simple graph $G$ is a proper colouring of $G$,
    we clearly have $\chi(G) \leq \chi_N(G)$.  Moreover, $\chi_N(G) \leq \#V(G)$ as the
    singleton colouring $\dcup_{v \in V(G)} \{v\}$ is nested.
\end{remark}

Graphs with very small or very large chromatic number are colour-nested.

\begin{lemma}\label{lem:1-n1-n}
    Let $G$ be a finite simple graph on $n$ vertices.  If $\chi(G) \in \{1, n-1, n\}$, where
    $n \geq 2$, then $\chi_N(G) = \chi(G)$, i.e., $G$ is colour-nested.
\end{lemma}
\begin{proof}
    Suppose that $\chi(G) = 1$.  Hence $E(G) = \emptyset$ and every vertex is a duplicate of every
    other vertex by Remark~\ref{rem:isolated-vertices}.  Thus the set $V(G)$ is a nested colouring
    of $G$ and $\chi_N(G) = \chi(G)$.

    Suppose that $\chi(G) = n-1$, where $n \geq 2$.  Thus $G$ is $K_n$ with a nonempty subset of
    the edges connected to some vertex, say, $v$, removed.  Let $u$ be a vertex nonadjacent to $v$.
    Thus $N_G(u) = V(G) \setminus \{u, v\}$ contains $N_G(v)$, and $v$ is a weak duplicate of $u$.
    Hence $\{u, v\}$ is a nested independent set of $G$ and so $\chi_N(G) \leq n-1$.  By the preceding
    remark we thus have $\chi_N(G) = \chi(G)$.

    Suppose that $\chi(G) = n$. By Remark~\ref{rem:singleton} we have
    $\chi(G) = \chi_N(G) = \#V(G)$.
\end{proof}

Moreover, the upper bound in Remark~\ref{rem:singleton} is sometimes attained by graphs with small
chromatic number.

\begin{example}\label{exa:petersen}
    Let $P$ be the Petersen graph; see Figure~\ref{fig:petersen}.  It is well-known that
    $\chi(P) = 3$.  However, since no vertex of $P$ is a weak duplicate of another vertex
    of $P$, $\chi_N(P) = 10 = \#V(P)$.
    \begin{figure}[!ht]
        \includegraphics[scale=1]{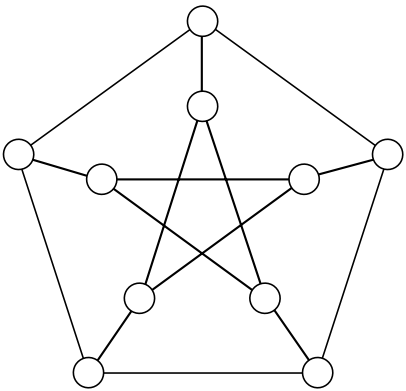}
        \caption{The Petersen graph.}
        \label{fig:petersen}
    \end{figure}
\end{example}

\begin{remark}\label{rem:sperner}
    Recall that a Sperner family is a collection of sets in which no set is a subset of another.
    Thus for a finite simple graph $G$, $\chi_N(G) = \#V(G)$ if and only if the set of open
    neighbourhoods of vertices of $G$ forms a Sperner family.  For example, the set of open
    neighbourhoods of vertices of the Petersen graph form a Sperner family.
\end{remark}

\subsection{The de-duplicate graph}\label{sub:de-dup}~

We now define a derivative graph based on the equivalence relation of duplicates.

\begin{definition}\label{def:de-dup}
    Let $G$ be a finite simple graph.  Define $\sim$ to be the equivalence relation of duplicates
    of $G$, and let $[\cdot]_\sim$ denote an equivalence class of this relation.  The
    \emph{de-duplicate graph} of $G$ is the graph $G^\star$ with vertices given by the equivalence
    classes $[v]_\sim$ for $v \in V(G)$ and with an edge between $[u]_\sim$ and $[v]_\sim$ if and only if
    $(u, v) \in E(G)$.
\end{definition}

Notice that $G \cong G^\star$ if and only if $G$ is duplicate-free.

Passing to the de-duplicate graph of $G$ does not change the nested chromatic number.

\begin{proposition}\label{pro:de-dup}
    If $G$ is a finite simple graph, then $\chi_N(G) = \chi_N(G^{\star})$.
\end{proposition}
\begin{proof}
    Suppose $\mC$ is the nested colouring $C_1 \ddd C_{k}$ of $G$.  For $1 \leq i \leq k$, let
    \[
        C'_i = \{[v]_{\sim} \st v \in C_i\} \setminus \cup_{j=1}^{i-1} C'_{j}.
    \]
    By construction, $C'_1 \ddd C'_{k}$ is a partition $\mC'$ of $V(G^\star)$.  Since $C_i$ is
    an independent set, $C'_i$ is as well.  Hence $\mC'$ is a proper colouring of $G^\star$.  Moreover,
    since $C_i$ is nested, $C'_i$ is nested with the order on the vertices of $C'_i$ induced by the order
    of the vertices of $C_i$, and so $\mC'$ is a nested colouring of $G^\star$.  Thus $\chi_N(G) \geq \chi_N(G^\star)$.

    On the other hand, suppose $\mathcal{D}$ is the nested colouring $D_1 \ddd D_{r}$ of $G^\star$.
    For $1 \leq i \leq r$, let $D'_i = \{v \st v \in V(G) \mbox{~and~} [v]_{\sim} \in D_i\}$.  Since
    $\mathcal{D}$ is a partition of $V(G^\star)$, $D'_1 \ddd D'_r$ is a partition $\mathcal{D}'$ of
    $V(G)$.  Moreover, since $D_i$ is an independent set, $D'_i$ is as well.  Hence $\mathcal{D}'$ is
    a proper colouring of $G$.  Since each $D_i$ is nested, $D'_i$ is nested with the order on the vertices
    of $D'_i$ induced by the order of the vertices of $D_i$, where the order on duplicate vertices is
    arbitrary.  Hence $\mathcal{D}'$ is a nested colouring of $G$.  Thus $\chi_N(G^\star) \geq \chi_N(G)$.
\end{proof}

Since a complete graph is the de-duplicate of a complete multipartite graph and the Tur\'an graph,
then its nested chromatic number is simple to compute.

\begin{corollary}\label{cor:complete-multi-partite}
    If $n_1, \ldots, n_r$ are positive integers, then $\chi_N(K_{n_1, \ldots, n_r}) = \chi(K_{n_1, \ldots, n_r}) = r$.
\end{corollary}

\begin{corollary}\label{cor:Turan}
    If $T_{n,r}$ is the Tur\'an graph on $n$ vertices that is $r$-partite, then $\chi_N(T_{n,r}) = \chi(T_{n,r}) = r$.
\end{corollary}

Moreover, duplicate-free graphs have been studied under various other names.

\begin{remark}\label{rem:dupfree}
    Duplicate-free graphs were studied by Sumner~\cite{Su} as ``point-determining graphs.''  Sumner
    showed that every connected point-determining graph has at least two vertices that can each be
    removed leaving point-determining induced subgraphs.

    They were also studied as ``mating graphs'' or ``$M$-graphs'' by Bull and Pease~\cite{BP} in order
    to understand mating-type systems.  In this case, vertices are identified with individuals in a
    population, and edges correspond to compatibility in mating.  Thus duplicate vertices correspond
    to individuals with identical mating compatibilities and so need not be represented.

    Kilibarda~\cite{Ki} proved a bijection between unlabeled (connected) mating graphs on $n$ vertices
    with unlabeled (connected) graphs without endpoints on $n$ vertices.  Thus~\cite[A004110]{OEIS}
    and~\cite[A004108]{OEIS} enumerate the number of unlabeled (connected) duplicate-free graphs on
    $n$ vertices.  We note that Kilibarda called the de-duplicate graph $G^\star$ the ``reduction of $G$.'

    Finally, duplicate-free graphs were used by McSorley~\cite{McS} as ``neighbourhood distinct graphs''
    to classify the neighbourhood anti-Sperner graphs, a related but distinct set of graphs.  A graph
    is \emph{neighbourhood anti-Sperner}, or \emph{NAS}, if every vertex is weakly duplicated by some
    other vertex.  Porter~\cite{Po} introduced the concept of NAS graphs, and showed that every NAS
    graph has a pair of duplicate vertices.  Porter and Yucas~\cite{PY} established more properties of
    NAS graphs.
\end{remark}

\subsection{The weak duplicate preorder}\label{sub:order}~

We next define a preorder on the vertices of a graph using the concept of weak duplicates.
It is particularly important to notice that the preorder is in the \emph{reverse} order of
containment.

\begin{definition}\label{def:weak-dup-preorder}
    Let $G$ be a finite simple graph.  The \emph{weak duplicate preorder} on $V(G)$
    is the preorder defined by $v \leq u$ if $u$ is a weak duplicate of $v$.
\end{definition}

Exchanging a vertex of a clique for a lesser vertex in the preorder generates another
clique of the graph.

\begin{lemma}\label{lem:clique-exchange}
    Let $G$ be a finite simple graph, and let $C$ be a clique of $G$.  If $u$ is a
    vertex of $C$, and $v \leq u$ under the weak duplicate preorder on $V(G)$, then
    $(C \dcup \{v\}) \setminus \{u\}$ is a clique of $G$.
\end{lemma}
\begin{proof}
    Since $N_G(u) \subset N_G(v)$, $C \subset N_G(u)$ implies $C \subset N_G(v)$.
    Thus $(C \dcup \{v\}) \setminus \{u\}$ is a clique of $G$.
\end{proof}

This gives an alternate condition on a partition of the vertices that is equivalent
to being a nested colouring.

\begin{proposition}\label{pro:clique-exchange}
    Let $G$ be a finite simple graph. If $\mC = C_1 \ddd C_k$ is a partition of $V(G)$,
    then the following conditions are equivalent:
    \begin{enumerate}
        \item $\mC$ is nested;
        \item there is an ordering on the vertices of each colour class $C_i$ such
            that if $v$ is less than $u$ in that order, and $K$ is a clique of $G$
            containing $u$, then $(K \dcup\{v\}) \setminus \{u\}$ is a clique of $G$;
            and
        \item there is an ordering on the vertices of each colour class $C_i$ such
            that if $v$ is less than $u$ in that order, and $\{u, w\}$ is an edge
            of $G$, then $\{v, w\}$ is an edge of $G$.
    \end{enumerate}
\end{proposition}
\begin{proof}
    Suppose that condition (i) holds.  Since each independent set $C_i$ is nested, the vertices
    of $C_i$ are comparable under the weak duplicate preorder.  If we arbitrarily order
    the duplicates in $C_i$, then the induced order on $C_i$ is the desired order for
    condition (ii) by Lemma~\ref{lem:clique-exchange}.

    Clearly, condition (ii) implies condition (iii), since edges are cliques of $G$.

    Suppose now that condition (iii) holds.  Since $\{u, w\} \in E(G)$ implies
    that $\{v, w\} \in E(G)$,  $N_G(u)$ is a subset of $N_G(v)$.  That is,
    the order on the vertices of $C_i$ respects the weak duplicate preorder, and so
    $C_i$ is nested.  In particular, condition (i) holds.
\end{proof}

When the graph is duplicate-free, the preorder is a partial order.

\begin{definition}\label{def:weak-dup-poset}
    Let $G$ be a duplicate-free finite simple graph.  The weak duplicate preorder on
    $G$ is then a partial order, and we write $P_G$ for the poset on $V(G)$ under
    the weak duplicate partial order induced by $G$.
\end{definition}

The key observation is that, when $G$ is duplicate-free, the chain covers of $P_G$ are
in bijection with the nested colourings of $G$.

\begin{proposition}\label{pro:chain-covers}
    Let $G$ be a duplicate-free finite simple graph.  A partition $C_1 \ddd C_k$ of
    $V(G)$ is a nested colouring of $G$ if and only if it is a chain cover of $P_G$.
\end{proposition}
\begin{proof}
    This follows from the definitions of a nested independent set and the weak duplicate partial order.
    In particular, $N_G(u) \subset N_G(v)$ if and only if $v \leq u$, and in
    both cases the sets of vertices form a partition of $V(G)$.
\end{proof}

Dilworth~\cite[Theorem~1.1]{Di} proved that the width (or Dilworth number) of a poset $P$, i.e.,
the maximum cardinality of an antichain of $P$, is precisely the minimum cardinality of a chain
cover of $P$.  Hence the nested chromatic number of a graph is the width of the poset of the
de-duplicate of the graph.

\begin{corollary}\label{cor:width}
    If $G$ is a finite simple graph, then $\chi_N(G)$ is the width of $P_{G^\star}$.
\end{corollary}

\begin{remark}\label{rem:dilworth}
    Let $G$ be a finite simple graph.  A vertex $v$ of $G$ \emph{dominates} a vertex $u$ of $G$
    if $N_G(u) \subset N_G[v]$.  Notice the subtle difference between domination and weak duplication,
    namely, $u$ and $v$ may be adjacent in the former.  The \emph{Dilworth number of $G$} is the cardinality of the
    largest set of vertices of $G$ such that no vertex dominates any other in the set.

    Following Felsner, Raghavan, and Spinrad~\cite{FRS}, we partially order the vertices of a duplicate-free
    graph $G$ by $v \leq u$ if $v$ dominates $u$.  The width of this partial order is precisely the Dilworth
    number of the graph $G$.  This partial order is in the \emph{reverse} order of containment, as in the weak
    duplicate partial order.

    The Dilworth number of a graph is \emph{not} the nested chromatic number of the graph despite
    the similarities.  Recall that threshold graphs are precisely the graphs with Dilworth number $1$.
    In Corollary~\ref{cor:threshold} we classify the nested chromatic number of threshold graphs as one
    more than the number of domination steps in the construction of the graph.
\end{remark}

As a consequence, the nested chromatic number can be computed in polynomial time.

\begin{theorem}\label{thm:poly-time}
    The nested chromatic number of a finite simple graph on $n$ vertices can be computed in $O(n^3)$ time.
\end{theorem}
\begin{proof}
    Fulkerson~\cite{Fu} proved that computing the width of a poset is equivalent to computing the cardinality
    of a maximum matching of a related bipartite graph.  Hopcroft and Karp~\cite{HK} proved that computing
    the latter can be done in $O(n^{5/2})$ time.

    Computing the relations between the $n$ vertices corresponds to computing $\binom{n}{2}$ subset
    containments, where each subset has size at most $O(n)$.  Hence computing the poset structure
    on $P_{G^\star}$ takes $O(n^3)$ time.  Thus computing the nested chromatic number of a finite simple
    graph on $n$ vertices via the width of the weak duplicate partial order takes $O(n^3)$ time.
\end{proof}

\begin{remark}\label{rem:computability}
    Since the nested chromatic number of a graph is the width of an associated poset, existing tools
    can be used to compute the value for specific cases.  Indeed, the computer algebra system
    \emph{Macaulay2}~\cite{M2} handles posets with the package \emph{Posets}~\cite{CMW}, which can
    compute the width of a poset.  Furthermore, using the package \emph{Nauty}~\cite{Co-Nauty}, one can generate
    all the graphs on a small number of vertices (with specific restrictions, e.g., bipartite only, if desired).
    The latter package uses the software \texttt{nauty}~\cite{Mc} at its core.

    The ease of computing the nested chromatic number on all graphs of small size is very helpful when
    proving results such as Theorem~\ref{thm:possible-chi-chi-s}.
\end{remark}

The poset $P_G$ need not be unique; see Figure~\ref{fig:non-unique}.

\begin{figure}[!ht]
    \begin{minipage}[b]{0.32\linewidth}
        \centering
        \includegraphics[scale=0.75]{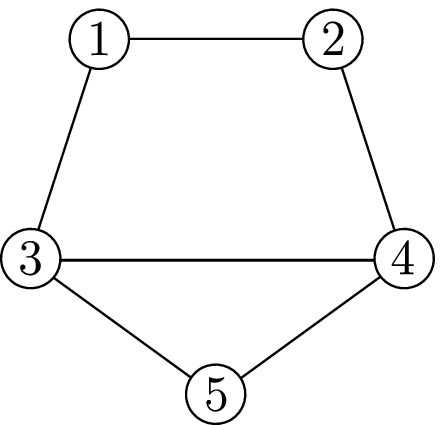}\\
        \emph{(i) The graph $G$.}
    \end{minipage}
    \begin{minipage}[b]{0.32\linewidth}
        \centering
        \includegraphics[scale=0.75]{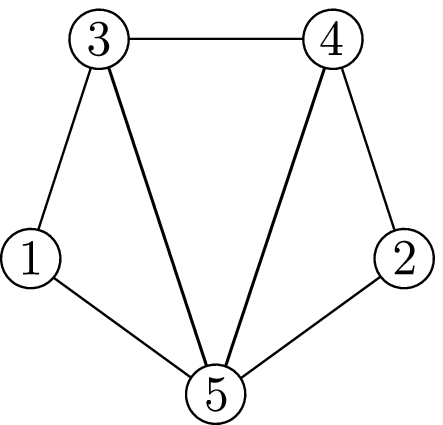}\\
        \emph{(ii) The graph $H$.}
    \end{minipage}
    \begin{minipage}[b]{0.32\linewidth}
        \centering
        \includegraphics[scale=0.75]{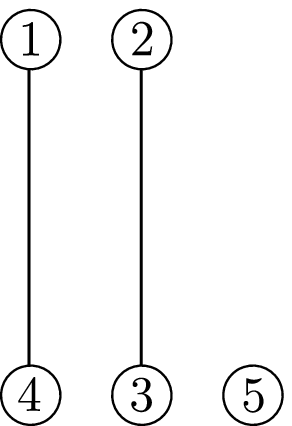}\\
        \emph{(iii) The poset $P_G = P_H$.}
    \end{minipage}
    \caption{The graphs $G$ and $H$ are non-isomorphic, but $P_G = P_H$.}
    \label{fig:non-unique}
\end{figure}

Furthermore, the poset need not be ranked; see Figure~\ref{fig:non-ranked}.

\begin{figure}[!ht]
    \begin{minipage}[b]{0.48\linewidth}
        \centering
        \includegraphics[scale=0.75]{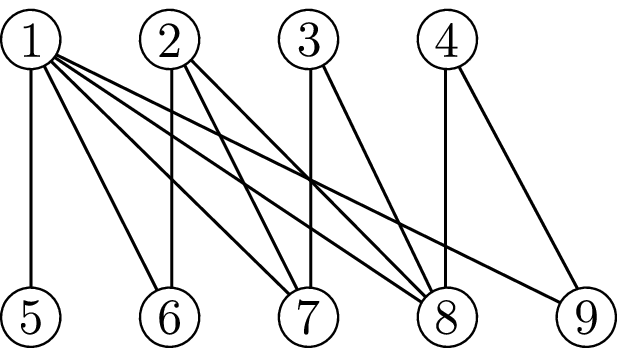}\\
        \emph{(i) The graph $G$.}
    \end{minipage}
    \begin{minipage}[b]{0.48\linewidth}
        \centering
        \includegraphics[scale=0.75]{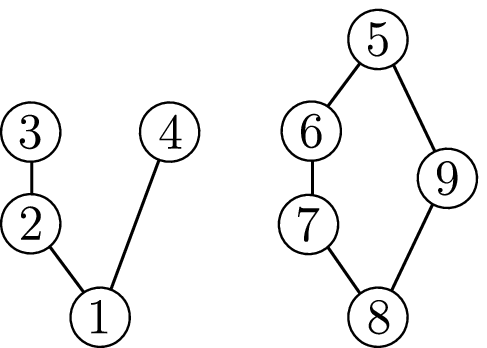}\\
        \emph{(ii) The poset $P_G$.}
    \end{minipage}
    \caption{The weak duplicate poset $P_G$ need not be ranked.}
    \label{fig:non-ranked}
\end{figure}

However, the height of the poset, i.e., the length of the longest chain, is restricted to at most
half the number of vertices.

\begin{proposition}\label{pro:height}
    If $G$ is a duplicate-free finite simple graph on $n$ vertices, then the height of $P_G$
    is at most $\left\lfloor \frac{n-1}{2} \right\rfloor$.
     That is, at most half of the vertices of $G$ can be in a nested independent set of $G$.
\end{proposition}
\begin{proof}
    Suppose the height of $P_G$ is $h$, i.e., there exist $h+1$ vertices $v_0, \ldots, v_h$ of $G$ such that
    $N_G(v_h) \subsetneq \cdots \subsetneq N_G(v_0)$.  This implies $N_G(v_0) \geq h$, and since $v_i \notin N_G(v_0)$
    for $0 \leq i \leq h$, $n \geq 2h+1$.  That is, $\frac{n-1}{2} \geq h$.
\end{proof}

\begin{example}\label{exa:not-weak-dup-poset}
    The preceding proposition implies that a poset with a large height relative to the number of vertices, e.g., a chain,
    cannot be the poset associated to a finite simple graph under the weak duplicate partial order.  However, there exist
    posets with small height that are also not associated to a finite simple graph.

    Consider the poset $P$ on $\{1,2,3,4\}$ with covering relations $1 < 2$, $1 < 3$, and $1 < 4$.
    This poset is not associated to a graph, as determined by a search of the $11$ graphs on $4$ vertices.
    Notice, however, that the dual of $P$ is the poset associated to the graph $K_3 \dcup K_1$.
\end{example}

\begin{question}\label{que:possible-posets}
    What posets are isomorphic to some $P_G$, where $G$ is a duplicate-free finite simple graph?
\end{question}

\section{Families of graphs}\label{sec:families}

We now look at three families of graphs with well-behaved nested chromatic numbers.

\subsection{Bipartite graphs}\label{sub:bip}~

Due to the structure of bipartite graphs it is possible to classify the graphs with
nested chromatic number $2$, i.e., colour-nested bipartite graphs.

\begin{theorem}\label{thm:bip}
    Let $r$ and $s$ be positive integers, and let $1 \leq a_r \leq \cdots \leq a_1 \leq s$ be
    a sequence of nonnegative integers.  Construct the graph $G = G_{a_1, \ldots, a_r; s}$ on
    the vertex set $\{u_1, \ldots, u_r, v_1, \ldots, v_s\}$ with an edge between $u_i$ and $v_j$
    if and only if $j \leq a_i$.  Then the following statements are true:
    \begin{enumerate}
        \item $\{u_1, \ldots, u_r\} \dcup \{v_1, \ldots, v_s\}$ is a nested colouring of $G$,
        \item $\chi_N(G) = \chi(G) = 2$, i.e., $G$ is colour-nested, and
        \item every nontrivial finite simple bipartite graph that is colour-nested arises this way.
    \end{enumerate}
\end{theorem}
\begin{figure}[!ht]
    \includegraphics[scale=0.9]{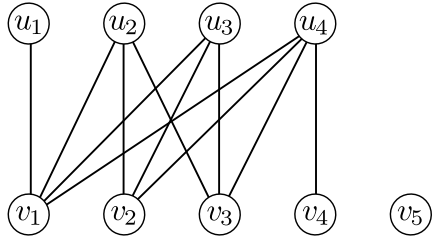}
    \caption{The colour-nested bipartite graph $G_{4,3,3,1;5}$.}
    \label{fig:bip-4331-5}
\end{figure}
\begin{proof}
    Let $U = \{u_1, \ldots, u_r\}$ and $V = \{v_1, \ldots, v_s\}$.

    By construction $N_G(u_i) = \{v_1, \ldots, v_{a_i}\}$ and hence $N_G(v_j) = \{u_i \st a_i \geq j\}$.
    Thus $U \dcup V$ is a colouring of $G$.  Since $a_1 \geq 1$, $u_1$ and $v_1$ are adjacent,
    hence $\chi(G) = 2$.  Since $a_{i+1} \leq a_i$, $N_G(u_{i+1})$ is a subset of $N_G(u_i)$ and 
    $N_G(v_{i+1})$ is a subset of $N_G(v_i)$.  Thus $U \dcup V$ is a nested colouring of $G$, and
    so $\chi_N(G) = 2$.  This completes parts (i) and (ii).

    Now let $G$ be any nontrivial finite simple bipartite graph that is colour-nested.  Suppose
    $\{u_1, \ldots, u_r\} \dcup \{v_1, \ldots v_s\}$ is a nested colouring of $G$, such that
    $N_G(u_{i+1}) \subset N_G(u_i)$ and $N_G(v_{i+1}) \subset N_G(v_i)$.  Furthermore, by
    Remark~\ref{rem:isolated-vertices} we may assume without loss of generality that any isolated
    vertices of $G$ are in $\{v_1, \ldots v_s\}$.  Set $a_i = \max\{j \st v_j \in N_G(u_i)\}$.
    Since $N_G(u_{i+1}) \subset N_G(u_i)$, $1 \leq a_r \leq \cdots \leq a_1 \leq s$.
    Thus $G$ arises as in the construction above, completing part (iii).
\end{proof}

The connected and duplicate-free colour-nested bipartite graphs have a simple classification.

\begin{corollary}\label{cor:bip-structure}
    Let $r$ and $s$ be positive integers, and let $1 \leq a_r \leq \cdots \leq a_1 \leq s$ be
    a sequence of nonnegative integers.  The following statements are true.
    \begin{enumerate}
          \item $G_{a_1, \ldots, a_r; s}$ is connected if and only if $a_1 = s$, and
        \item $G_{a_1, \ldots, a_r; s}$ is connected and duplicate-free if and only if $r = s$ and $a_i = i$, for $1 \leq i \leq r$.
    \end{enumerate}
\end{corollary}
\begin{proof}
    Let $G = G_{a_1, \ldots, a_r; s}$, $U = \{u_1, \ldots, u_r\}$, and $V = \{v_1, \ldots, v_s\}$.

    Since $a_r \geq 1$, we have $N_G(v_1) = U$.  As $U \dcup V$ is a nested colouring of $G$, 
    $G$ is connected if and only if $N_G(u_1) = V$, i.e., $a_1 = s$.  This completes part (i).

    Furthermore, $N_G(u_i) = N_G(u_j)$ if and only if $a_i = a_j$, and $N_G(v_i) = N_G(v_j)$ if and only
    if $\max\{k \st a_k \geq i\} = \max\{k \st a_k \geq j\}$, i.e., there exists a $k$ such that
    $a_k - a_{k+1} \geq 2$.  This completes part (ii).
\end{proof}

Further, this permits an enumeration of certain colour-nested bipartite graphs.

\begin{corollary}\label{cor:enum-bip}
    If $n \geq 3$ is an odd integer, then there are precisely $2^{n-3}$ unique connected colour-nested
    bipartite graphs with $n$ vertices.

    Furthermore, there exists a unique duplicate-free colour-nested bipartite graph with $n$ vertices,
    where $n \geq 2$ is an integer.
\end{corollary}
\begin{proof}
    By Theorem~\ref{thm:bip} and Corollary~\ref{cor:bip-structure}, we need only look at graphs
    $G_{a_1, \ldots, a_r; s}$ where $1 \leq a_r \leq \cdots \leq a_1 = s$ such that $r + s = n$.

    For fixed $r$ and $s$, $G_{a_1, \ldots, a_r; s} = G_{a'_1, \ldots, a'_r; s}$
    if and only if $a_i = a'_i$; otherwise the vertex degree sequences of the $u_i$ would differ.
    Hence there are $\binom{n-2}{r-1}$ graphs (we choose with repetition the $r-1$ values
    $a_1, \ldots, a_{r-1}$ among the $s$ options) with the specified $r$ and $s$.  Thus among all
    choices of $r$ and $s$ there are $\sum_{r=1}^{n-1}\binom{n-2}{r-1} = 2^{n-2}$ possible graphs.
    However, if $r > s$, then $G_{a_1, \ldots, a_r; s} = G_{b_1, \ldots, b_s; r}$ where
    $b_i = \max\{j \st a_j \geq i \}$ for $1 \leq i \leq s$.  Thus we have double counted in our
    enumeration, and so there are $2^{n-3}$ unique graphs.

    The second statement follows by Corollary~\ref{cor:bip-structure}(ii).  In particular, if $n=2m$,
    then $G_{1,\ldots,m;m}$ is the unique duplicate-free graph, and if $n = 2m+1$, then $G_{1,\ldots,m;m+1}$
    is the unique duplicate-free graph.
\end{proof}

\subsection{Regular graphs}\label{sub:regular}~

The nested chromatic number of a graph is the same as the number of vertices if and only if the graph is duplicate-free.
Moreover, large girth can force a regular graph to be duplicate-free.

\begin{proposition}\label{pro:regular}
    Let $G$ be a $d$-regular finite simple graph, for $d \geq 1$.  The graph $G$ is duplicate-free if and only if $\chi_N(G) = \#V(G)$.

    In particular, if the girth of $G$ is at least $5$, then $\chi_N(G) = \#V(G)$.
\end{proposition}
\begin{proof}
    Let $G$ be a $d$-regular finite simple graph.  Since $\#N_G(u) = d$ for all vertices $u$ of $G$, 
    $u$ is a weak duplicate of $v$ if and only if $u$ and $v$ are duplicates.

    Now suppose $G$ has girth at least $5$.  If $u$ and $v$ are distinct duplicates, then $u$ and $v$ have at least
    two common neighbours, say, $\{w,x\}$.  Thus either $\{u,v,w,x\}$ induces a $4$-cycle or $\{v,w,x\}$ induces a
    $3$-cycle, i.e., the girth of $G$ is at most $4$.  This contradicts the girth of $G$ being at least $5$, and so
    $G$ is duplicate-free.
\end{proof}

As an immediate consequence, we can compute the nested chromatic number of snarks and Kneser graphs.
See Example~\ref{exa:petersen} for the Petersen graph, which is both a snark and the Kneser graph $KG_{5,2}$.

\begin{corollary}\label{cor:snark}
    If $G$ is a snark, then $\chi_N(G) = \#V(G)$.
\end{corollary}
\begin{proof}
    Snarks are $3$-regular and have girth at least $5$.
\end{proof}

\begin{corollary}\label{cor:kneser}
    If $n$ and $k$ are positive integers so that $n \geq 2k$, then the nested chromatic number of the Kneser
    graph $KG_{n,k}$ is $\chi_N(KG_{n,k}) = \#V(KG_{n,k}) = \binom{n}{k}$.
\end{corollary}
\begin{proof}
    Recall that the vertices of the Kneser graph $KG_{n,k}$ are the $k$-subsets of $\{1, \ldots, n\}$, and
    a pair of vertices are adjacent if the corresponding sets are disjoint.  This implies that no two vertices
    are duplicates, otherwise they would be the same $k$-subset.  By Proposition~\ref{pro:regular}, $KG_{n,k}$
    being duplicate-free implies that $\chi_N(KG_{n,k}) = \#V(KG_{n,k})$.
\end{proof}

Let $\overline{G}$ denote the complement of the finite simple graph $G$.  The nested chromatic number
of the $n$-cycle $C_n$ and the $n$-anticycle $\overline{C_n}$ are simple expressions, for large $n$.

\begin{corollary}\label{cor:cycle}
    Let $n \geq 3$ be an integer.  The following statements are true:
    \begin{enumerate}
        \item $\chi_N(C_3) = 3$ and $\chi_N(\overline{C_3}) = 1$,
        \item $\chi_N(C_4) = 2$ and $\chi_N(\overline{C_4}) = 4$, and
        \item $\chi_N(C_n) = n = \chi_N(\overline{C_n})$, for $n \geq 5$.
    \end{enumerate}
\end{corollary}
\begin{proof}
    Parts (i) and (ii) are easy to verify.  Since $C_n$ has girth $n$ and is $2$-regular, by
    Proposition~\ref{pro:regular}, $\chi_N(C_n) = n$ for $n \geq 5$.

    Let $n \geq 5$.  The $n$-anticycle $\overline{C_n}$ is $(n-3)$-regular.  Suppose $u$ and $v$
    are distinct vertices of $\overline{C_n}$ such that $u$ is a duplicate of $v$.  This implies
    that there is a vertex $w$, distinct from $u$ and $v$, that is nonadjacent to $u$ and $v$.
    Hence $\{u,v,w\}$ is an independent set in $\overline{C_n}$ and so induces a $3$-cycle in $C_n$,
    which is absurd.  Thus $\overline{C_n}$ is duplicate-free and $\chi_N(\overline{C_n}) = n$
    by Proposition~\ref{pro:regular}.
\end{proof}

This further emphasises the distinction between the chromatic number and the nested chromatic number.

\begin{remark}\label{rem:comp-planar}
    Since $C_n$ is a planar graph, this shows that planar graphs can have arbitrarily large nested
    chromatic number.  This contrasts the chromatic number for planar graphs, which is bounded by $4$.
    See Proposition~\ref{pro:planar} for more about the nested chromatic number and planar graphs.

    Let $G$ be a finite simple graph, and let $\overline{G}$ denote the complement of $G$.
    In this case, $\chi(G) + \chi(\overline{G}) \leq \#V(G)+1$.  However, the nested
    chromatic number can break this bound.  Indeed, by the previous lemma, we have
    $\chi_N(C_n) + \chi_N(\overline{C_n}) = 2n = 2\#V(C_n)$ for $n \geq 5$.
    On the other hand, $\chi_N(P_4) + \chi_N(\overline{P_4}) = \#V(P_4) = 4$, since $\overline{P_4} \cong P_4$.
\end{remark}

We offer a conjecture suggested by the preceding remark.

\begin{conjecture}\label{con:complement}
    If $G$ is a finite simple graph, then $\chi_N(G) + \chi_N(\overline{G}) \geq \#V(G)$.
\end{conjecture}

\subsection{Diamond- and \texorpdfstring{$C_4$}{four-cycle}-free graphs}\label{sub:weakly-geodetic}~

Let the \emph{diamond graph} be $K_4$ with any edge removed; see Figure~\ref{fig:diamond-C4}(i).  If $G$
is both diamond- and $C_4$-free, then only the presence of leaves, i.e., degree $1$ vertices, can
reduce the nested chromatic number from $\#V(G)$.

\begin{figure}[!ht]
    \begin{minipage}[b]{0.48\linewidth}
        \centering
        \includegraphics[scale=1]{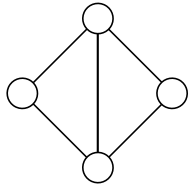}\\
        \emph{(i) The diamond graph.}
    \end{minipage}
    \begin{minipage}[b]{0.48\linewidth}
        \centering
        \includegraphics[scale=1]{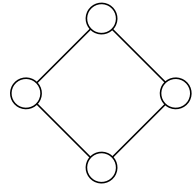}\\
        \emph{(ii) The $4$-cycle $C_4$.}
    \end{minipage}
    \caption{The forbidden graphs in Section~\ref{sub:weakly-geodetic}.}
    \label{fig:diamond-C4}
\end{figure}

\begin{theorem}\label{thm:weakly-geodetic}
    Let $G$ be a connected finite simple graph that is both diamond- and $C_4$-free.  If $G$ has
    $\ell$ leaves, then $\#V(G) - \ell \leq \chi_N(G) \leq \#V(G)$.  Furthermore,
    equality holds in the upper bound if and only if either $\ell = 0$ or $G = K_2$.

    In particular, if the minimum degree of a vertex $\delta(G)$ is at least $2$, then
    $\chi_N(G) = \#V(G)$.
\end{theorem}
\begin{proof}
    Suppose $G$ is a connected finite simple graph that is both diamond- and $C_4$-free, and further
    suppose $G$ has $\ell$ leaves.

    Let $u$ and $v$ be distinct vertices of $G$.  If $u$ and $v$ have two neighbours in common, say, $w$
    and $x$, then $\{u, v, w, x\}$ must be a $4$-clique of $G$ since it cannot be a diamond or a $C_4$;
    thus $u$ and $v$ must be adjacent.  Hence if $u$ is a weak duplicate of $v$, then $u$ and $v$ must have
    exactly one neighbour in common, and so $\#N_G(u) = 1$, i.e., $u$ is a leaf.  Thus at most one element
    of each nested independent set is a non-leaf, and so $\chi_N(G)$ is at least the number of non-leaves, i.e.,
    $\#V(G) - \ell$.

    Clearly, if $G$ has no leaves, then $\chi_N(G) = \#V(G)$.  Suppose $G$ is not $K_2$, and $G$ has at
    least one leaf, say, $u$.  Let $v$ be the unique neighbour of $v$.  As $G$ is connected and not $K_2$,
    $v$ must have at least one neighbour not $u$, say, $w$.  Hence $N_G(u) = \{v\} \subset N_G(w)$, and $u$
    is a weak duplicate of $w$.  Thus $\{u, w\}$ is a nested independent set of $G$, and so $\chi_N(G) < \#V(G)$.
\end{proof}

Clearly, graphs with girth at least $5$ are diamond- and $C_4$-free.  Since $d$-regular graphs have no leaves,
if $d \geq 2$, then this recovers the second part of Proposition~\ref{pro:regular}, as well as
Corollary~\ref{cor:snark} and Corollary~\ref{cor:cycle}(iii).

Trees, which have infinite girth, are diamond- and $C_4$-free graphs.

\begin{corollary}\label{cor:trees}
    Let $G$ be a finite simple tree with at least three vertices.  If $G$ has $\ell$ leaves, then
    $\#V(G) - \ell \leq \chi_N(G) < \#V(G)$.
\end{corollary}

This immediately gives the nested chromatic number for path graphs.

\begin{corollary}\label{cor:path}
    Let $P_n$ be the path graph on $n$ vertices.  The nested chromatic number of $P_n$ is
    \[
        \chi_N(P_n) =
        \left\{
            \begin{array}{ll}
                2 & \mbox{if $2 \leq n \leq 4$,} \\
                4 & \mbox{if $n = 5$, and } \\
                n-2 & \mbox{if $n \geq 6$.}
            \end{array}
        \right.
    \]
\end{corollary}
\begin{proof}
    If $2 \leq n \leq 5$, then it is simple to verify the claim.

    Suppose $n \geq 6$, and without loss of generality assume the edges of $P_n$ are
    $\{\{1,2\},\ldots,\{n-1,n\}\}$.  In this case, $N_{P_n}(1) = \{2\} \subset N_{P_n}(3) = \{2,4\}$
    and $N_{P_n}(n) = \{n-1\} \subset N_{P_n}(n-2) = \{n-3,n-1\}$.  Since $n \geq 6$, $n-2 \neq 3$, and so
    \[
        \{1,3\} \dcup \{n-2, n\} \dcup  \{2\} \dcup \{4\} \ddd \{n-3\} \dcup \{n-1\}
    \]
    is a nested colouring of $P_n$.  Hence $\chi_N(P_n) \leq n-2$, and so equality holds by
    Corollary~\ref{cor:trees}.
\end{proof}

We close with some comments about the class of diamond- and $C_4$-free graphs.

\begin{remark}\label{rem:weakly-geodetic}
    The class of diamond- and $C_4$-free graphs has been studied in the more general setting of
    diamond- and even-cycle-free graphs by Kloks, M\"uller, and Vu\v{s}kovi\'c~\cite{KMV}.  Some of their
    results specify to the case of diamond- and $C_4$-free graphs.

    In a more focused case, Eschen, Ho\`ang, Spinrad, and Srithavan~\cite{EHSS} studied structural
    results on this class of graphs.  Moreover, they provide a polynomial-time recognition algorithm.
    They make use of an alternate classification of diamond- and $C_4$-free graphs:  they are precisely
    the graphs such that every nonadjacent pair of vertices has at most one common neighbour.

    We further note that diamond- and $C_4$-free graphs were called \emph{weakly geodetic graphs} in the past;
    see, e.g., \cite{KC}.
\end{remark}

\section{Induced subgraphs}\label{sec:induced-subgraphs}

A first natural operation to consider is that of taking induced subgraphs.

\subsection{Induced subgraphs}\label{sub:induced-subgraphs}~

The nested chromatic number behaves the same as the chromatic number under
taking induced subgraphs.

\begin{proposition}\label{pro:induced-subgraph}
    Let $G$ be a finite simple graph, and let $H$ be an induced subgraph of $G$.
    If $C_1 \ddd C_k$ is a nested colouring $\mC$ of $G$, then
    $(C_1 \cap V(H)) \ddd (C_k \cap V(H))$ is a nested colouring $\mC'$ of $H$.

    In particular, $\chi_N(H) \leq \chi_N(G)$.
\end{proposition}
\begin{proof}
    It is already known that $\mC'$ is a proper colouring of $H$, since
    $\mC$ is a proper colouring of $G$.  Moreover, since $N_H(v) = N_G(v) \cap V(H)$
    for $v \in V(H)$, the nesting of $C_i$ implies the nesting of $C_i \cap V(H)$.
\end{proof}

Together with Corollary~\ref{cor:cycle}, the preceding proposition implies that the maximum length of
an induced cycle, if one exists and is big enough, forms an effective lower bound for the nested
chromatic number.

\begin{corollary}\label{cor:cycle-bound}
    Let $G$ be a finite simple graph which has at least one induced cycle.  If the maximum length of
    an induced cycle $c$ is at least $5$, then $\chi_N(G) \geq c$.
\end{corollary}
\begin{proof}
    This follows from Proposition~\ref{pro:induced-subgraph} and Corollary~\ref{cor:cycle}.
\end{proof}

\begin{remark}\label{rem:cycle-bound}
    We offer a pair of comments about the preceding corollary.
    \begin{enumerate}
        \item If the girth of a graph is finite and at least $5$, then it is a lower bound for
            the nested chromatic number of the graph.
        \item Suppose the longest induced \emph{odd} cycle of $G$ has length $2k-1$.  Erd\H{o}s
            and Hajnal~\cite[Theorem~7.7]{EH} proved that $\chi(G) \leq 2k$.  On the other hand,
            the preceding result shows that $\chi_N(G) \geq 2k-1$, if $k \geq 3$.  See~\cite{KV}
            for further results bounding $\chi(G)$ using the length of the longest induced odd cycle.
    \end{enumerate}
\end{remark}

Let $G$ be a finite simple graph, and let $v$ be a vertex of $G$.  The \emph{vertex deletion of $G$ by $v$}
is the induced subgraph $G - v$ of $G$ on vertex set $V(G) \setminus \{v\}$.  The chromatic number is reduced by
at most one after vertex deletion.  The nested chromatic number is reduced by at most one more than the
degree of the vertex that was deleted.

\begin{proposition}\label{pro:vertex-deletion}
    Let $G$ be a finite simple graph.  If $v$ is any vertex of $G$, then
    \[
        \chi_N(G) - \#N_G(v) - 1 \leq \chi_N(G - v) \leq \chi_N(G).
    \]
\end{proposition}
\begin{proof}
    The upper bound follows immediately from Proposition~\ref{pro:induced-subgraph}.

    Let $C_1 \ddd C_k$ be a nested colouring of $G - v$.  This implies that
    $C'_1 \ddd C'_k$, where $C'_i = C_i \setminus N_G(v)$, together with $\{v\}$ and the singleton
    sets containing each neighbour of $v$ is a nested colouring of $G$.  This follows as
    the presence of $v$ only affects the neighbourhoods of its neighbours.  Hence
    $k + \#N_G(v) + 1 \geq \chi_N(G)$, and so $k \geq \chi_N(G) - \#N_G(v) - 1$.
\end{proof}

Both bounds in the preceding proposition are achievable.

\begin{example}\label{exa:vertex-deletion}
    Let $n \geq 3$.  Notice that $C_n - v = P_{n-1}$ and $\#N_{C_n}(v) = 2$ for any vertex $v$ of $C_n$.
    Combining Corollaries~\ref{cor:cycle} and~\ref{cor:path}, we have that
    $\chi_N(C_n - v) = \chi_N(P_{n-1}) = \chi_N(C_n) - \#N_{C_n}(v) - 1$ if $n \geq 5$ and $n \neq 6$.

    \begin{figure}[!ht]
        \includegraphics[scale=0.75]{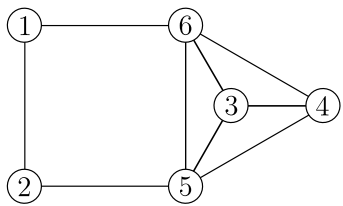}
        \caption{A graph $G$ such that $\chi_N(G) = \chi_N(G - 5) = \chi_N(G - 6) = 4$.}
        \label{fig:vertex-deletion}
    \end{figure}
    On the other hand, let $G$ be as in Figure~\ref{fig:vertex-deletion}.  We have
    $\chi_N(G) = \chi_N(G - 5) = \chi_N(G - 6) = 4$ despite $\#N_G(5) = \#N_G(6) = 4$.
\end{example}

\subsection{Criticality}\label{sub:criticality}~

Recall that a vertex $v$ of a finite simple graph $G$ is a \emph{critical vertex} if $\chi(G - v) = \chi(G) - 1$.
Further, if every vertex of $G$ is a critical vertex, then $G$ is \emph{vertex-critical} (or \emph{vertex-colour-critical}).
Critical vertices are never weak duplicates of other vertices in $G$.

\begin{lemma}\label{lem:critical-weak-dup}
    Let $G$ be a finite simple graph.  If $v$ is a critical vertex of $G$, and $v$ is a weak duplicate of
    $w \in G$, then $w = v$.
\end{lemma}
\begin{proof}
    Suppose $\chi(G) = k$, and let $C_1 \ddd C_{k-1}$ be an optimal colouring of $G - v$.
    Assume $w \neq v$ and $w \in C_1$, without loss of generality.  Since $N_G(v) \subset N_G(w)$,
    $v$ is independent of the vertices in $C_1$.  Hence $(C_1 \dcup \{v\}) \dcup C_2 \ddd C_{k-1}$
    is a colouring of $G$, contradicting $\chi(G) = k$.  Thus $w = v$.
\end{proof}

This implies that the number of critical vertices provides a lower bound for the nested chromatic number.

\begin{corollary}\label{cor:critical-bound}
    Let $G$ be a finite simple graph.  If $c$ is the number of critical vertices of $G$, then
    $\chi_N(G) \geq c$.

    In particular, if $G$ is vertex-critical, then $\chi_N(G) = \#V(G)$.
\end{corollary}
\begin{proof}
    By Lemma~\ref{lem:critical-weak-dup}, every critical vertex of $G$ must be at the top of its
    own nested colour class, which immediately implies the bound.
\end{proof}

We define a concept of criticality for the nested chromatic number.

\begin{definition}\label{def:nested-critical}
    A finite simple graph $G$ is \emph{nested-critical} if the deletion of any vertex reduces
    the nested chromatic number of $G$.
\end{definition}

Graphs with large nested chromatic number are nested-critical.

\begin{proposition}\label{pro:stab-crit}
    Let $G$ be a finite simple graph.  If $\chi_N(G) = \#V(G)$, then $G$ is nested-critical.

    In particular, if $G$ is vertex-critical, then $G$ is nested-critical.
\end{proposition}
\begin{proof}
    This follows immediately since $\chi_N(G - v) \leq \#V(G - v) < \chi_N(G)$.

    The second claim follows from Corollary~\ref{cor:critical-bound}.
\end{proof}

Being nested-critical does not imply being vertex-critical.  For example,
$P_n$ is nested-critical for $n \geq 7$ but is never vertex-critical.

However, if $G$ is colour-nested, then being nested-critical is equivalent to being vertex-critical.

\begin{proposition}\label{pro:colour-nested-critical}
    Let $G$ be a finite simple graph.  If $G$ is colour-nested, then the following conditions are equivalent:
    \begin{enumerate}
        \item $G$ is nested-critical,
        \item $G$ is vertex-critical,
        \item $\chi_N(G) = \#V(G)$, and
        \item $G = K_{\#V(G)}$.
    \end{enumerate}
\end{proposition}
\begin{proof}
    Since $\chi(G) = \chi_N(G)$, and $\chi(H) \leq \chi_N(H)$ in general, we clearly have condition (i)
    implying condition (ii), and the latter implies condition (iii) by Corollary~\ref{cor:critical-bound}.
    Condition (iii) implies $\chi(G) = \#V(G)$, which is equivalent to $G = K_{\#V(G)}$.
    Finally, $K_{\#V(G)}$ is nested-critical by Proposition~\ref{pro:stab-crit}.
\end{proof}

On the other hand, if $G$ is colour-nested, then $G - v$ need not be colour-nested.  The graph $G$ in
Figure~\ref{fig:vertex-deletion} is colour-nested, though $G - 5$ and $G - 6$ are not colour-nested.

\subsection{A topological remark}\label{sub:topology}~

Consider the \emph{homomorphism complex} $\Hom(H, G)$ coming from the homomorphisms from
$H$ to $G$, where $G$ and $H$ are finite simple graphs; see~\cite[Definition~3.2]{Ko}.  
Kozlov showed~\cite[Theorem~3.3]{Ko} that $\Hom(H, G)$ collapses onto $\Hom(H, G - u)$ if
$u$ is a weak duplicate of some other vertex $v$ in $G$.  Thus if $C_1 \ddd C_k$ is a nested
$k$-colouring of $G$, and $G'$ is an induced subgraph of $G$ with vertex set $\{v_1, \ldots, v_k\}$,
where $v_i$ is a minimal element of $C_i$ under the weak duplicate preorder, then $\Hom(H, G)$
collapses onto $\Hom(H, G')$, and so $\Hom(H, G)$ and $\Hom(H, G')$ have the same simple homotopy type.

This is particularly interesting as the neighbourhood complex of $G$, i.e., the simplicial
complex of subsets of $V(G)$ which have a common neighbour, is homotopy equivalent to $\Hom(K_2, G)$.
Thus Lov\'asz's lower bound on the chromatic number~\cite[Theorem~2]{Lo} can be interpreted as
$\conn \Hom(K_2, G) \leq \chi(G) - 3$, where $\conn{X}$ is the connectivity of the complex $X$.
We note that this lower bound is strict in the case of Kneser graphs.

Hence we see that there exists an induced subgraph $G'$ of $G$ on $\chi_N(G)$ vertices
such that $\Hom(K_2, G') \leq \chi(G) - 3$.  Thus if $\chi_N(G) < \#V(G)$, then $G$ has more
redundancy than necessary for such topological bounds on the chromatic number to be useful.
Indeed, it is this redundant and recursive nature that is exploited in the associated 
algebras studied in~\cite{Co-UFI}.

Further, recall that the \emph{independence complex} $\Ind(G)$ of a finite simple graph $G$ is the
simplicial complex with faces given by the independent sets of $G$.  Engstr\"om showed~\cite[Lemma~3.2]{En}
that $\Ind(G)$ collapses onto $\Ind(G - v)$ if $v$ is weakly duplicated by some other vertex $u$.
More generally, this implies that if $C_1 \ddd C_k$ is a nested $k$-colouring of $G$, and $G''$ is an
induced subgraph of $G$ with vertex set $\{u_1, \ldots, u_k\}$, where $u_i$ is a maximal element
of $C_i$ under the weak duplicate preorder, then $\Ind(G)$ collapses onto $\Ind(G'')$, and so
$\Ind(G)$ and $\Ind(G'')$ have the same simple homotopy type.  Again, this emphasises the redundant
structure present in graphs with $\chi_N(G) < \#V(G)$.

We note that the de-duplicate graph $G^\star$ of $G$ can be constructed in both of these fashions.
In particular, let $C_1 \ddd C_k$ be the nested $k$-colouring such that each colour-class is an
equivalence class of duplicates.  Selecting an induced subgraph of $G$ with precisely one element
from each colour class generates a graph isomorphic to $G^\star$.  Hence $\Hom(H, G)$ and
$\Hom(H, G^\star)$ have the same simple homotopy type, as do $\Ind(G)$ and $\Ind(G^\star)$.

\section{Behaviour of the nested chromatic number}\label{sec:behaviour}

Now we consider the behaviour of the nested chromatic number under various graph operations.

\subsection{Mycielski's construction}\label{sub:mycielski}~

Let $G$ be a simple graph on $V(G) = \{u_1, \ldots, u_n\}$.  The \emph{Mycielski graph of $G$} is the
graph $\mu(G)$ with vertex set $\{u_1, \ldots, u_n, v_1, \ldots, v_n, w\}$ with edge set
\[
    E(\mu(G)) = E(G) \dcup \{ \{u_i, v_j\} \st u_j \in N_G(u_i)\} \dcup \{ \{w, v_i\} \st 1 \leq i \leq n \}.
\]
This construction was first described by Mycielski~\cite{My}, wherein he proved that $\chi(\mu(G)) = \chi(G) + 1$,
and further that $\mu(G)$ is triangle-free if $G$ is triangle-free.  We further note that $G$ is an induced
subgraph of $\mu(G)$.

Unlike the chromatic number, which only increases by one, the nested chromatic number doubles and increases by one
under the Mycielski construction.

\begin{proposition}\label{pro:mycielski}
    If $G$ is a finite simple graph, then $\chi_N(\mu(G)) = 2\chi_N(G) + 1$.
\end{proposition}
\begin{proof}
    The vertices of $\mu(G)$ have the open neighbourhoods:
    $N_{\mu(G)}(u_i) = \{ u_j, v_j \st u_j \in N_G(u_i) \}$, $N_{\mu(G)}(v_i) = N_G(u_i)$,
    and $N_{\mu(G)}(w) = \{v_1, \ldots, v_n\}$.

    Let $\mC$ be any nested colouring $C_1 \ddd C_k$ of $G$.  Each $C_i$ remains a nested independent
    set in $\mu(G)$.  Moreover, substituting $v_j$ for $u_j$ in each $C_i$ generates a nested independent set $C'_i$
    in $\mu(G)$.  Thus $C_1 \ddd C_k \dcup C'_1 \ddd C'_k \dcup \{w\}$ is a nested colouring of $\mu(G)$, and so
    $\chi_N(\mu(G)) \leq 2\chi_N(G)+1$.

    Isolated vertices of $G$ remain isolated in $\mu(G)$, and none of the $v_j$ nor $w$ can be isolated.
    If $u_i$ is not an isolated vertex of $G$, then $u_i$ is not a weak duplicate of any $v_j$ since the $v_j$ are
    not adjacent to any other $v_t$ in $\mu(G)$, and every $v_j$ is not a weak duplicate of $u_i$ since only the $v_j$
    are adjacent to $w$.  Moreover, with the exception of isolated vertices, $w$ is not a weak duplicate of or weakly
    duplicated by any other vertex.  Thus any nontrivial nested independent set of $\mu(G)$ that does not
    contain isolated vertices is contained exclusively in one of $\{u_1, \ldots, u_n\}$, $\{v_1, \ldots, v_n\}$, and $\{w\}$.

    Let $\mC$ be any nested colouring $C_1 \ddd C_k$ of $\mu(G)$.  By the preceding paragraph, we may assume without
    loss of generality that $C_1 \ddd C_i = \{u_1, \ldots, u_n\}$, $C_{i+1} \ddd C_{k-1} = \{v_1, \ldots, v_n\}$,
    and $C_k = \{w\}$.  Thus $C_1 \ddd C_i$ induces a nested colouring on $G$, and so $i \geq \chi_N(G)$.
    Similarly, substituting $u_j$ for $v_j$ in $C_{i+1} \ddd C_{k-1}$, we have another nested colouring of $G$, and so
    $k - i - 1 \geq \chi_N(G)$.  Hence $k \geq 2\chi_N(G) + 1$.
\end{proof}

\begin{example}\label{exa:mycielski}
    In~\cite{My}, Mycielski presented the family $M_i$ recursively defined by $M_2 = K_2$ and $M_{k+1} = \mu(M_k)$,
    for $k \geq 2$.  Since $M_2$ is a triangle-free graph with $\chi(M_2) = 2$, $M_k$ is a triangle-free graph
    with $\chi(M_k) = k$.  For $2 \leq k \leq4$, $M_k$ is the triangle-free graph with fewest vertices such that
    $\chi(M_k) = k$.

    The nested chromatic number of $M_2 = K_2$ is $\#V(M_2) = 2$.  By Proposition~\ref{pro:mycielski}, it follows
    that $\chi_N(M_k) = \#V(M_k) = 2^{k-2} \cdot 3 - 1$.
\end{example}

\subsection{Disjoint union}\label{sub:disjoint}~

The chromatic number of a graph is the maximum of the chromatic numbers of the components of the graph.
The nested chromatic number, on the other hand, is additive along the components.

\begin{proposition}\label{pro:components}
    Let $G$ be a finite simple graph without isolated vertices, and let $G_1, \ldots, G_t$ be the
    components of $G$.  The partition $\mC$ of $V(G)$ is a nested colouring of $G$ if and only if
    $\mC = \mC_1 \ddd \mC_t$, where $\mC_1, \ldots, \mC_t$ are nested colourings of $G_1, \ldots, G_t$, respectively.

    In particular, $\chi_N(G) = \chi_N(G_1) + \cdots + \chi_N(G_t)$.
\end{proposition}
\begin{proof}
    If $v \in G_i$, then $N_{G_i}(v) = N_G(v)$.  Since there are no isolated vertices, none of these
    neighbourhoods are empty.  Hence no colour class can contain vertices from two separate
    components of $G$.  Furthermore, since the neighbourhoods do not change, the nesting of
    a colour class does not change when it is considered in $G$ or a component.
\end{proof}

Hence the disjoint union is also additive.

\begin{corollary}\label{cor:disjoint-union}
    If $G_1, \ldots, G_t$ are finite simple graphs without isolated vertices, then
    \[
        \chi_N(G_1 \ddd G_t) = \chi_N(G_1) + \cdots + \chi_N(G_t).
    \]

    In particular, if $G$ is a nontrivial finite simple graph and $t \geq 1$, then $\chi_N(tG) = t\chi_N(G)$.
\end{corollary}

Moreover, a graph with small nested chromatic number and no isolated vertices is connected.

\begin{corollary}\label{cor:chi-s-3}
    Let $G$ be a finite simple graph with no isolated vertices.  If $\chi_N(G) \leq 3$,
    then $G$ is connected.
\end{corollary}

\subsection{Join}\label{sub:join}~

Let $G$ and $H$ be finite simple graphs.  The \emph{join of $G$ and $H$} is the graph
$G \vee H$ with vertex set $V(G) \dcup V(H)$, where all edges of $G$ and $H$ are preserved
and every vertex in $V(G)$ is adjacent to every vertex in $V(H)$.
In particular, the open neighbourhoods of $g$ in $V(G)$ and $h$ in $V(H)$ are
\[
    N_{G \vee H}(g) = N_G(g) \dcup V(H)~\mbox{~and~}~N_{G \vee H}(h) = N_H(h) \dcup V(G),
\]
respectively.

Both the chromatic number and the nested chromatic number are additive across joins.

\begin{proposition}\label{pro:join}
    Let $G$ and $H$ be finite simple graphs.  The partition $\mC$ of $V(G) \dcup V(H)$
    is a nested colouring of $G \vee H$ if and only if $\mC = \mC_G \dcup \mC_H$, where
    $\mC_G$ and $\mC_H$ are nested colourings of $G$ and $H$, respectively.

    In particular, $\chi_N(G \vee H) = \chi_N(G) + \chi_N(H)$.
\end{proposition}
\begin{proof}
    As every vertex of $G$ is adjacent to every vertex of $H$, no colour class can contain
    vertices from both $G$ and $H$.  Moreover, since all the vertices of $G$ (resp., $H$)
    have their neighbourhoods modified in a uniform way, nesting is not changed.
\end{proof}

This implies, in particular, that adding a dominating vertex, i.e., a vertex adjacent to
every other vertex, to a graph increases the nested chromatic number by precisely $1$.

\begin{example}\label{exa:dominating}
    Many common families of graphs are constructed by adding a dominating vertex to another
    common graph.  Consider the following examples.
    \begin{enumerate}
        \item The star graph $S_n$ is the trivial graph on $n$ vertices with a dominating
            vertex added.  Hence $\chi_N(S_n) = 2$ for $n \geq 1$.
        \item The windmill graph $W{\kern -.2em}d_{k,n}$ is $n K_k$ with a dominating vertex added.
            Hence \[\chi_N(W{\kern -.2em}d_{k,n}) = \chi_N(n K_k) + 1 = n \chi_N(K_k) + 1 = nk + 1.\]
        \item The wheel graph $W_n$ is the cycle graph $C_n$ with a dominating vertex added.
            Hence $\chi_N(W_n) = n+1$ for $n = 3$ and $n \geq 5$, and $\chi_N(W_4) = 3$,
            by Corollary~\ref{cor:cycle}.
    \end{enumerate}
\end{example}

Further, threshold graphs are colour-nested.  Recall that a threshold graph is a graph that
can be constructed from a single isolated vertex by repeatedly adding a new isolated vertex
or a new dominating vertex.

\begin{corollary}\label{cor:threshold}
    If $G$ is a threshold graph constructed with $d$ dominating steps, then
    $\chi_N(G) = \chi(G) = d+1$.
\end{corollary}
\begin{proof}
    Adding isolated vertices does not change the nested chromatic number as seen in
    Remark~\ref{rem:isolated-vertices}.  By Proposition~\ref{pro:join}, adding
    a dominating vertex increases the nested chromatic number by $1$.  Hence the
    nested chromatic number of $G$ is one more than the number of dominating steps.
    That is $\chi_N(G) = d+1$.  Moreover, after $d$ dominating steps, the clique
    number of $G$, $\omega(G)$, is $d+1$.  Since $\omega(G) \leq \chi(G) \leq \chi_N(G)$,
    we have $\chi(G) = d+1$.
\end{proof}

\subsection{Direct product}\label{sub:direct-product}~

Let $G$ and $H$ be finite simple graphs.  The \emph{direct} (or \emph{tensor}) \emph{product of $G$ and $H$}
is the graph $G \times H$ with vertex set $V(G) \times V(H)$, where $(g,h)$ is adjacent to $(g',h')$ if and
only if $g$ is adjacent to $g'$ in $G$ and $h$ is adjacent to $h'$ in $H$.  In particular, the
open neighbourhood of $(g,h)$ in $G \times H$ is
\[
    N_{G \times H}(g,h) = N_G(g) \times N_H(h).
\]
Notice that $G \times K_1 \cong \overline{K_n}$, so $\chi_N(G \times K_1) = 1$.  Moreover,
$(G \dcup G') \times H = (G \times H) \dcup (G' \times H)$.  Thus following Section~\ref{sub:disjoint},
we need only consider finite simple graphs $G$ that are connected and have at least two vertices.

Let $(P, \leq_P)$ and $(Q, \leq_Q)$ be posets.  The \emph{direct product of $P$ and $Q$} is the poset
$(P \times Q, \leq_{P \times Q})$, where $(p, q) \leq_{P \times Q} (p', q')$ if and only if
$p \leq_P p'$ and $q \leq_Q q'$.  The weak duplicate poset of the direct product of two graphs is
the direct product of the weak duplicate posets of the graphs.

\begin{lemma}\label{lem:dirprod-poset}
    If $G$ and $H$ are duplicate-free connected finite simple graphs, neither of which is $K_1$,
    then $P_{G \times H} = P_G \times P_H$.
\end{lemma}
\begin{proof}
    Since $N_{G \times H}(g,h) = N_G(g) \times N_H(h)$, it is immediate that
    $(g, h)$ is a weak duplicate of $(g', h')$ in $G \times H$ if and only if
    $g$ is a weak duplicate of $g'$ in $G$ and $h$ is a weak duplicate of $h'$ in $H$.
    The claim follows immediately.
\end{proof}

As a consequence, the nested chromatic number of the direct product of graphs is bounded below by
the product of the nested chromatic numbers of the factors.  We note that the chromatic
number of the direct product of graphs is bounded above by the minimum of the chromatic numbers
of the factors (Hedetniemi's conjecture says equality holds).

\begin{proposition}\label{pro:dirprod}
    If $G$ and $H$ are connected finite simple graphs, neither of which is $K_1$, then
    \[
        \chi_N(G) \cdot \chi_N(H) \leq \chi_N(G \times H) \leq \min\{\#V(G) \cdot \chi_N(H), \chi_N(G) \cdot \#V(H)\}.
    \]

    In particular, if $\chi_N(H) = \#V(H)$, then $\chi_N(G \times H) = \chi_N(G) \cdot \chi_N(H)$.
\end{proposition}
\begin{proof}
    By Proposition~\ref{pro:de-dup} we may assume $G$ and $H$ are duplicate-free.  Thus
    by Lemma~\ref{lem:dirprod-poset} we have that $P_{G \times H} = P_G \times P_H$.  Hence
    $\chi_N(G \times H)$ is the width of $P_G \times P_H$, by Corollary~\ref{cor:width}.

    Clearly, if $A$ and $B$ are antichains of $P_G$ and $P_H$, respectively, then $A \times B$
    is an antichain of $P_{G \times H}$.  Hence the width of $P_{G \times H}$ is at least the
    product of the widths of $P_G$ and $P_H$, i.e., $\chi_N(G \times H) \geq \chi_N(G) \cdot \chi_N(H)$.

    On the other hand, let $A$ be any antichain of $P_{G \times H}$.  For each $g$ in $P_G$, let
    $A_g = \{h \in P_H \st (g, h) \in A\}$.  By construction, $A_g$ must be an antichain of $P_H$ for
    all $g \in P_G$.  This implies that $\#A_g \leq \chi_N(H)$ and so $\chi_N(G \times H) \leq \#V(G) \cdot \chi_N(H)$.
    As the graph direct product is commutative, we also then have $\chi_N(G \times H) \leq \chi_N(G) \cdot \#V(H)$
    by symmetry.
\end{proof}

Both bounds are achievable.

\begin{example}\label{exa:dirprod}
    Let $G = P_4$, and let $H$ be the graph on $V(H) = \{1,2,3,4\}$ with edge set $E(H) = \{\{1,2\},\{1,3\},\{2,3\},\{3,4\}\}$.
    In this case, $\chi_N(G) = 2$ and $\chi_N(H) = 3$.  However, $\chi_N(G \times H) = 8 = \chi_N(G) \cdot \#V(H)$.

    On the other hand, equality holds with the lower bound for the bipartite double cover $G \times K_2$ of $G$.  In particular,
    this implies that the crown graph on $2n$ vertices, i.e., $K_n \times K_2$, has nested chromatic number $2n$.
\end{example}

\subsection{Cartesian product}\label{sub:cartesian-product}~

Let $G$ and $H$ be finite simple graphs.  The \emph{Cartesian product of $G$ and $H$} is the graph
$G \gprod H$ with vertex set $V(G) \times V(H)$, where $(g,h)$ is adjacent to $(g',h')$ if and only
if either $g = g'$ and $h$ is adjacent to $h'$ in $H$ or $h = h'$ and $g$ is adjacent to $g'$ in $G$.
In particular, the open neighbourhood of $(g,h)$ in $G \gprod H$ is
\[
    N_{G \gprod H}(g,h) = \{g\} \times N_H(h) \dcup N_G(g) \times \{h\}.
\]

Notice that $G \gprod K_1$ is isomorphic to $G$, so $\chi_N(G \gprod K_1) = \chi_N(G)$.  Moreover,
$(G \dcup G') \gprod H = (G \gprod H) \dcup (G' \gprod H)$.  Thus following Section~\ref{sub:disjoint},
we need only consider finite simple graphs $G$ that are connected and have at least two vertices.

The weak duplicates generated in the Cartesian product come from leaves.

\begin{lemma}\label{lem:prod-dup}
    Let $G$ and $H$ be connected finite simple graphs, neither of which is $K_1$.
    The vertex $(g, h)$ is a weak duplicate of the distinct vertex $(g', h')$ in $G \gprod H$
    if and only if $N_G(g) = \{g'\}$ and $N_H(h) = \{h'\}$.
\end{lemma}
\begin{proof}
    By definition, $(g, h)$ is a weak duplicate of $(g', h')$ if and only if
    \[
        N_{G \gprod H}(g, h) = \{g\} \times N_H(h) \dcup N_G(g) \times \{h\} \subset
        N_{G \gprod H}(g', h') = \{g'\} \times N_H(h') \dcup N_G(g') \times \{h'\}.
    \]

    If $g = g'$, then $(g, h)$ being a weak duplicate of $(g', h')$ forces $h = h'$, since $N_G(g)$ is
    nonempty and open.  Hence we may assume $g \neq g'$ and $h \neq h'$.  In this case, $(g, h)$ is a
    weak duplicate of $(g', h')$ if and only if $\{g\} \times N_H(h) \subset N_G(g') \times \{h'\}$ and
    $N_G(g) \times \{h\} \subset \{g'\} \times N_H(h'\}$, since $N_G(g)$ and $N_H(h)$ are nonempty and
    open.  The latter is equivalent to $N_H(h) = \{h'\}$ and $N_G(g) = \{g'\}$, again since the
    neighbourhoods are nonempty.
\end{proof}

Thus except $K_2 \gprod K_2 = C_4$, all Cartesian products of connected graphs are duplicate-free.

\begin{corollary}\label{cor:prod-dupfree}
    Let $G$ and $H$ be connected finite simple graphs, neither of which is $K_1$.
    The graph $G \gprod H$ is duplicate-free if and only if $G \neq K_2$ or $H \neq K_2$.
\end{corollary}
\begin{proof}
    By Lemma~\ref{lem:prod-dup}, $(g, h)$ is a duplicate of the distinct vertex $(g', h')$
    if and only if $N_G(g) = \{g'\}$, $N_G(g') = \{g\}$, $N_H(h) = \{h'\}$, and $N_H(h') = \{h\}$,
    i.e., $G = H = K_2$.
\end{proof}

Moreover, we can compute the nested chromatic number of Cartesian products of connected graphs.
Whereas the chromatic number of the Cartesian product of graphs is the maximum of the chromatic
numbers of the factors, the nested chromatic number is close to the product of the nested chromatic
numbers of the factors.  We recall that $[v]_\sim$ is the equivalence class of duplicate vertices in
$G$, defined in Definition~\ref{def:de-dup}.

\begin{proposition}\label{pro:cart-prod}
    Let $G$ and $H$ be connected finite simple graphs, neither of which is $K_1$.
    If $G \neq K_2$ or $H \neq K_2$, then
    \[
        \chi_N(G \gprod H) = \#V(G \gprod H) - \ell'(G) \cdot \ell'(H),
    \]
    where $\ell'(L) = \#\{[v]_\sim \st v \mbox{~is a leaf of $L$}\}$ for a graph $L$.

    In particular, if $G$ or $H$ has minimum vertex degree of at least $2$, then
    $\chi_N(G \gprod H) = \#V(G \gprod H)$.
\end{proposition}
\begin{proof}
    By Lemma~\ref{lem:prod-dup}, $\{(g, h), (g', h')\}$ is a nested independent set of
    $G \gprod H$ if and only if $N_G(g) = \{g'\}$ and $N_H(h) = \{h'\}$.  Since $G \gprod H$
    is duplicate-free by Corollary~\ref{cor:prod-dupfree}, no colour class can contain
    more than two vertices.

    Let $L$ be the set of all colour classes of two vertices.  Every nested colouring of $G \gprod H$
    consists of subset of $L$ of pairwise disjoint elements together with singleton sets of the
    remaining vertices.  In particular, $\chi_N(G \gprod H)$ is $\#V(G \gprod H)$ minus the
    largest subset of $L$ that consists of pairwise disjoint elements.

    Since the weak duplicate vertex of each element of $L$ is unique, selection of a subset of
    $L$ of pairwise disjoint elements depends only on the weakly duplicated vertex of each element
    of $L$.  In particular, if $(g, h)$ is the weak duplicate vertex of an element of $L$, then no
    other element of $L$ with weak duplicate vertex $(i,j)$ such that $[g]_\sim = [i]_\sim$ and
    $[h]_\sim = [j]_\sim$ can be in such a disjoint set.  Thus the largest subset of $L$
    that consists of pairwise disjoint elements is of size $\ell'(G) \cdot \ell'(H)$.
\end{proof}

\begin{example}\label{exa:cube}
    The cube graph $Q_n$ is defined recursively by $Q_1 = K_2$ and $Q_n = Q_{n-1} \gprod K_2$, so $Q_n$
    has no leaves for $n \geq 2$.  Hence $\chi_N(Q_n) = 2^n$ if $n \neq 2$ and $\chi_N(Q_2) = 2$.
    This also follows by Proposition~\ref{pro:regular} since $Q_n$ is $n$ regular and duplicate-free for
    $n \neq 2$ by Corollary~\ref{cor:prod-dupfree}.
\end{example}

\subsection{Strong product}\label{sub:strong-product}~

Let $G$ and $H$ be finite simple graphs.  The \emph{strong product of $G$ and $H$} is the graph
$G \boxtimes H$ with vertex set $V(G) \times V(H)$, where $(g,h)$ is adjacent to the distinct
vertex $(g',h')$ if and only if $g = g'$ or $g$ is adjacent to $g'$ in $G$, and $h = h'$ and
$h$ is adjacent to $h'$ in $H$.  In particular, the open neighbourhood of $(g,h)$
in $G \boxtimes H$ is
\[
    N_{G \boxtimes H}(g,h) = N_G[g] \times N_H[h] \setminus \{(g, h)\}.
\]

Notice that $G \boxtimes K_1$ is isomorphic to $G$, so $\chi_N(G \boxtimes K_1) = \chi_N(G)$.  Moreover,
$(G \dcup G') \boxtimes H = (G \boxtimes H) \dcup (G' \boxtimes H)$.  Thus following Section~\ref{sub:disjoint},
we need only consider finite simple graphs $G$ that are connected and have at least two vertices.

With the exception of $G \boxtimes K_1$, the strong product of connected graphs has no weak duplicate vertices.

\begin{lemma}\label{lem:strong-dup}
    Let $G$ and $H$ be connected finite simple graphs, neither of which is $K_1$.
    The vertices $(g, h)$ and $(g', h')$ are weak duplicates in $G \boxtimes H$ if and only
    if $(g, h) = (g',h')$.
\end{lemma}
\begin{proof}
    Since $G$ and $H$ are connected, $N_G(g) \neq \emptyset \neq N_H(h)$.

    Suppose $(g, h)$ is a weak duplicate of $(g', h')$.  This implies that
    $\{g\} \times N_H(h) \subset N_{G \boxtimes H}(g', h')$, and so $g \in N_G[g']$, i.e.,
    $g' \in N_G[g]$.  By symmetry, we also have $h' \in N_H[h]$.  If $(g, h) \neq (g', h')$,
    then $(g', h') \in N_{G \boxtimes H}(g, h) \subset N_{G \boxtimes H}(g', h')$, which is
    absurd.
\end{proof}

Thus the nested chromatic number of the strong product is the number of vertices of the product.

\begin{proposition}\label{pro:strong}
    If $G$ and $H$ are connected finite simple graphs, neither of which is $K_1$, then
    $\chi_N(G \boxtimes H) = \#V(G) \cdot \#V(H)$.
\end{proposition}

\subsection{Composition}\label{sub:composition}~

Let $G$ and $H$ be finite simple graphs.  The \emph{composition} (or \emph{lexicographic product})
\emph{of $G$ and $H$} is the graph $G[H]$ with vertex set $V(G) \times V(H)$, where $(g,h)$ is adjacent
to $(g',h')$ if and only if either $g$ is adjacent to $g'$ in $G$ or $g = g'$ and $h$ is adjacent to $h'$ in $H$.
In particular, the open neighbourhood of $(g,h)$ in $G[H]$ is
\[
    N_{G[H]}(g,h) = N_G(g) \times V(H) \dcup \{g\} \times N_H(h).
\]
Clearly, composition is non-commutative, in general.

The weak duplicates in the composition come from weak duplicates of the operands.

\begin{lemma}\label{lem:comp-dup}
    Let $G$ and $H$ be finite simple graphs.  The vertex $(g,h)$ is a weak duplicate of the distinct
    vertex $(g', h')$ in $G[H]$ if and only if either $g = g'$ and $h$ is a weak duplicate of $h'$ in $H$
    or $g$ is a weak duplicate of $g'$ in $G$ and $h$ is an isolated vertex in $H$.
\end{lemma}
\begin{proof}
    Suppose $g = g'$.  This implies that $(g,h)$ is a weak duplicate of $(g', h')$ if and only
    if $N_H(h) \subset N_H(h')$, i.e., $h$ is a weak duplicate of $h'$ in $H$.

    Assume $g \neq g'$.  Further suppose $h$ is an isolated vertex in $H$.  Thus
    $N_{G[H]}(g, h) = N_G(g) \times V(H)$, and so $(g,h)$ is a weak duplicate of $(g',h')$
    if and only if $N_G(g) \subset N_G(g')$, i.e., $g$ is a weak duplicate of $g'$ in $G$.

    Now suppose $h$ is not an isolated vertex in $H$, and suppose $(g, h)$ is a weak
    duplicate of $(g', h')$.  This implies that $g$ is adjacent to $g'$; hence
    $\{g'\} \times V(H) \subset N_{G[H]}(g', h')$, i.e., $V(H) \subset N_H(h')$,
    which is absurd.
\end{proof}

From this we can derive conditions classifying which compositions are duplicate-free.

\begin{corollary}\label{cor:comp-dupfree}
    Let $G$ and $H$ be finite simple graphs.  The graph $G[H]$ is duplicate-free
    if and only if $H$ is duplicate-free and either
    \begin{enumerate}
        \item $H$ has no isolated vertices, or
        \item $H$ has an isolated vertex and $G$ is duplicate-free.
    \end{enumerate}
\end{corollary}
\begin{proof}
    Let $(g, h)$ and $(g', h')$ be distinct vertices of $G[H]$.

    Suppose $g = g'$.  By Lemma~\ref{lem:comp-dup}, $(g,h)$ and
    $(g', h')$ are duplicates in $G[H]$ if and only if $h$ and $h'$ are duplicates in $H$.

    Now suppose $g \neq g'$.  By Lemma~\ref{lem:comp-dup}, $(g,h)$
    and $(g', h')$ are duplicates in $G[H]$ if and only if $h$ and $h'$ are isolated
    vertices in $H$ and $g$ and $g'$ are duplicates in $G$.
\end{proof}

Further, we can bound the nested chromatic number of a graph composition, and equality
holds when the secondary graph has no isolated vertices.

\begin{proposition}\label{pro:comp}
    If $G$ and $H$ are finite simple graphs, then
    \[
        \chi_N(G[H]) \leq \#V(G) \cdot \chi_N(H).
    \]

    Moreover, equality holds if $H$ has no isolated vertices.
\end{proposition}
\begin{proof}
    Let $\mC$ be a nested colouring $C_1 \ddd C_k$ of $H$.  For each
    $g \in V(G)$ and for $1 \leq i \leq k$, set $C_{i,g} = \{g\} \times C_i$.
    By Lemma~\ref{lem:comp-dup}, $C_{i,g}$ is a nested independent set of $G[H]$, and so the family
    $C_{i,g}$ forms a nested colouring of $G[H]$.  Hence $\chi_N(G[H]) \leq \#V(G) \cdot \chi_N(H)$.

    Assume $H$ has no isolated vertices.  If $\{(g_1,h_1), \ldots, (g_t,h_t)\}$ is a nested independent
    set of $G[H]$, then by Lemma~\ref{lem:comp-dup} $g_1 = \cdots = g_t$ and $\{h_1, \ldots, h_t\}$
    forms a nested independent set of $H$.  Thus any nested colouring of $G[H]$ is of the form
    described in the first paragraph, and so $\chi_N(G[H]) = \#V(G) \cdot \chi_N(H)$.
\end{proof}

We suspect the following question has a negative answer.

\begin{question}\label{que:comp}
    Does there exist a pair of finite simple graphs $G$ and $H$ such that
    $\chi_N(G[H]) < \#V(G) \cdot \chi_N(H)$?
\end{question}

\subsection{Monotonicity}\label{sub:monotone}~

Recall that a graph property is \emph{monotone decreasing} (\emph{monotone increasing}, respectively)
if it is preserved under deletion (respectively, addition) of edges.  For example,
removing an edge can only decrease the chromatic number of a graph, so being $k$-colourable
is a monotone decreasing graph property.  However, having a nested $k$-colouring is neither
monotone decreasing or increasing.  To see this, we use three of the graph products discussed above.

Let $G$ and $H$ be finite simple graphs, and suppose $\chi_N(H) < \#V(H)$.  By construction,
we know that $G \times H$ is $G \boxtimes H$ with edges removed and $G \boxtimes H$ is likewise
$G[H]$ with edges removed.  That is,
\[
    E(G \times H) \subset E(G \boxtimes H) \subset E(G[H]) \subset E(K_{\#V(G) \cdot \#V(H)}).
\]
However, by Propositions~\ref{pro:dirprod} and~\ref{pro:comp}, both
$\chi_N(G \times H)$ and $\chi_N(G[H])$ are at most $\#V(G) \cdot \chi_N(H) < \#V(G) \cdot \#V(H)$.  Hence
using Proposition~\ref{pro:strong} we have that
\[
    \chi_N(G \times H) < \chi_N(G \boxtimes H) > \chi_N(G[H]) < \chi_N(K_{\#V(G) \cdot \#V(H)}).
\]

\section{On the existence of graphs}\label{sec:exists}

Given integers $c$ and $n$ such that $1 \leq c \leq n$, it is known that there
exists a finite simple graph $G$ on $n$ vertices with $\chi(G) = c$.  We
show that if we are also given an integer $s$ such that $1 \leq c \leq s \leq n$, then
$G$ can be chosen so that $\chi_N(G) = s$ for all but a few specific cases.

For fixed $n \geq 2$, the case when $c \in \{1, n-1, n\}$ was handled in Lemma~\ref{lem:1-n1-n}.
The one other infinite case is that there does not exist a bipartite graph with nested
chromatic number $3$.

\begin{lemma}\label{lem:no-bip-3}
    If $G$ is a bipartite graph, then $\chi_N(G) \neq 3$.
\end{lemma}
\begin{proof}
    Let $G$ be a bipartite graph, and suppose, without loss of generality (see
    Remark~\ref{rem:isolated-vertices}), that $G$ has no isolated vertices.  Suppose
    $\chi_N(G) \leq 3$. Hence by Corollary~\ref{cor:chi-s-3} we may assume $G$ is
    connected, and so $G$ has a unique proper $2$-colouring $B \dcup W$.

    We may assume without loss of generality that $W$ is a nested independent set of $G$.
    Let $w_1, \ldots, w_t$ be the elements of $W$ ordered such that $N_G(w_{i+1}) \subset N_G(w_i)$.
    Thus for each $b \in B$ there is a $k$ such that $b \in N_G(w_i)$ if and only if $1 \leq i \leq k$,
    i.e., $N_G(b) = \{w_1, \ldots, w_k\}$.  Hence $B$ is also nested, and $\chi_N(G) = 2$.
\end{proof}

We are ready to give the classification.

\begin{theorem}\label{thm:possible-chi-chi-s}
    Let $c$, $s$, and $n$ be integers such that $1 \leq c \leq s \leq n$.  There does not exist
    a finite simple graph $G$ on $n$ vertices with $\chi(G) = c$ and $\chi_N(G) = s$ if
    and only if one of the following conditions holds:
    \begin{enumerate}
        \item $c = 1$ and $s > 1$,
        \item $c = 2$ and $s = 3$,
        \item $c = 2$ and $(n, s)$ is one of $(4,4)$, $(5,5)$, $(6,5)$, and $(7,7)$, or
        \item $c = n-1$ and $s = n$.
    \end{enumerate}

    Moreover, if such a graph $G$ exists, then it may be chosen to be connected.
\end{theorem}
\begin{proof}
    If $c = s = n = 1$, then $G = K_1$.  Suppose $n \geq 2$.  If $c = 1$, $c = n-1$, or $c = n$, then by Lemma~\ref{lem:1-n1-n}
    there exists a finite simple graph $G$ on $n$ vertices with $\chi(G) = c$ and $\chi_N(G) = s$ if and only if $s = c$.
    Hence we may also suppose $2 \leq c \leq n-2$.

    By Lemma~\ref{lem:no-bip-3} if $c = 2$, then $s \neq 3$.  Moreover, checking the $143$ bipartite
    graphs with between $4$ and $7$ vertices shows that if $(n, s)$ is one of $(4,4)$, $(5,5)$, $(6,5)$, and $(7,7)$,
    then there is no finite simple graph $G$ on $n$ vertices with $\chi(G) = c$ and $\chi_N(G) = s$.
    Thus the conditions (i)--(iv) each imply the absence of the desired graph.

    Moreover, checking the $1251$ simple graphs with between $2$ and $7$ vertices, we see that,
    except for the conditions (i)--(iv), the desired (connected) simple graphs do indeed exist.

    To show the presence of the desired graphs in the remaining case, we proceed by induction on the number of vertices $n$.

    \emph{Base case:}  Suppose $n = 8$.  Checking the $12346$ ($11117$ of which are connected) simple
    graphs on $8$ vertices, we see that there exists a (connected) simple graph with $\chi(G) = c$ and $\chi_N(G) = s$
    for $2 \leq c \leq s \leq 6$, with the exception of $c = 2$ and $s = 3$.

    \emph{Inductive step:}  Suppose $n \geq 9$.  By induction, there exists a connected simple graph $G$ on $n-1$ vertices
    with $\chi(G) = c$ and $\chi_N(G) = s$ for $2 \leq c \leq s \leq n-1$, except for $(c,s) = (2,3)$ and $(c,s) = (n-2,n-1)$.
    If we duplicate any vertex of $G$, then the resulting connected graph $G'$ has $n$ vertices, $\chi(G') = \chi(G)$, and
    $\chi_N(G') = \chi(G)$ since the duplicate vertex can always be put in the same colour class as the duplicated vertex.
    If we add a dominating vertex to $G$, then the resulting connected graph $G''$ has $n$ vertices and $\chi(G'') = \chi(G) + 1$.
    Moreover, $\chi_N(G'') = \chi(G) + 1$ by Proposition~\ref{pro:join}.  Together these two operations generate the
    desired (connected) graph for all relevant $c$ and $s$, except $c = 2$ and $s = n$.

    If $n$ is even, then $\chi(C_n) = 2$ and $\chi_N(C_n) = n$, by Corollary~\ref{cor:cycle}.  If $n$ is odd, then consider
    the graph $H$ found by adding the vertex $0$ and the edges $\{0,1\}$ and $\{0,5\}$ to $C_{n-1}$.  Clearly, $H$ is a
    connected simple graph on $n$ vertices.  Moreover, $\chi(H) = 2$, since the partition of the vertices into even and odd
    vertices is a proper $2$-colouring of $H$.  Further still, the neighbourhoods of $H$ are: $N_H(0) = \{1,5\}$,
    $N_H(1) = \{0,n-1,2\}$, $N_H(5) = \{0,4,6\}$, $N_H(n-1) = \{1,n-2\}$, and $N_H(i) = \{i-1,i+1\}$ for $1 < i < n-1$ and $i \neq 5$.
    Thus no two vertices of $H$ are weak duplicates, and so $\chi_N(H) = n$.
\end{proof}

See Remark~\ref{rem:computability} for comments about using computer algebra systems to determine the nested chromatic number
of a finite simple graph.

In Remark~\ref{rem:comp-planar}, we noted that the nested chromatic number for a planar graph need not
be bound above by four, as is the chromatic number.  Indeed, we show that every possible nested chromatic
number can occur for a connected planar graph.

\begin{proposition}\label{pro:planar}
    Let $n \geq 2$.  For $2 \leq k \leq n$, there exists a connected planar simple graph $G$ on
    $n$ vertices with $\chi_N(G) = k$.
\end{proposition}
\begin{proof}
    Let $G$ be the graph $K_k$ if $2 \leq k \leq 4$, otherwise let $G$ be the graph $C_k$ if $k \geq 5$.
    Then clearly $G$ is a connected planar graph with $\chi_N(G) = k$ by Lemma~\ref{lem:1-n1-n} or
    Corollary~\ref{cor:cycle}, respectively.

    Without loss of generality, let $V(G) = \{1, \ldots, k\}$, and suppose $k-1$ and $k$ are adjacent.
    Modify $G$ by adding $n-k$ new vertices $\{k+1, \ldots, n\}$ and $n-k$ new edges $\{k-1, i\}$, where
    $k+1 \leq i \leq n$, to create the graph $G'$.  Clearly, $G'$ is a connected planar graph, as the new
    vertices are all leaves on the planar graph $G$.  Further still, $\{1\} \ddd \{k-1\} \dcup \{k, \ldots, n\}$
    is a nested colouring of $G'$.  Thus $\chi_N(G') = \chi_N(G) = k$, by Proposition~\ref{pro:induced-subgraph}.
\end{proof}

\begin{acknowledgement}
    The author thanks Benjamin Braun for many helpful comments, including pointing 
    out the references that lead to the inclusion of Section~\ref{sub:topology}.
\end{acknowledgement}


\end{document}